\documentclass[12pt,reqno]{amsart}
\usepackage{amssymb}
\usepackage{enumerate}
\usepackage{amsthm}
\usepackage{tikz-cd}
\usepackage{amsmath}
\usepackage[mathscr]{euscript}
\usepackage{comment}
\usepackage{fixltx2e}
\usepackage{xcolor}
\usepackage{hyperref}

\DeclareMathOperator{\Nef}{{Nef}}
\DeclareMathOperator{\mult}{{mult}}

\makeatletter
\@namedef{subjclassname@2010}{%
  \textup{2010} Mathematics Subject Classification}
\makeatother
\newtheorem{thm}{Theorem}[section]

\newtheorem{prop}[thm]{Proposition}
\newtheorem{defi}[thm]{Definition}

\newtheorem{xrem}{Remark}

\begin{document}
\baselineskip=17pt

\author{Rupam Karmakar}
\author{Praveen Kumar Roy}
\address{Indian Statistical Institute 203 Barrackpore Trunk Road Kolkata 700 108, WB, India}
\email[Rupam Karmakar]{rupammath91@gmail.com}

\address{UM-DAE Centre for Excellence in Basic Sciences
``Nalanda'', Opp Nano Sciences Building,
University of Mumbai, Vidyanagari,
Mumbai 400098, India.}
\email[Praveen Kumar Roy]{praveen.roy@cbs.ac.in}

\title{Seshadri constants on some blow-ups of projective spaces}
\maketitle

\begin{abstract}
Let $X^n_{r,s}$ denote the blow-up of $\mathbb{P}^n$ along $r$ general lines and $s$ general points. In this paper, we focus on $l$-very ample line bundles on $X^n_{0,s}$ and investigate their Seshadri constants with some restrictions on $s$. Additionally, we compute the nef cone of $X^3_{r,0}$ for $r\leq 6$ and study the Seshadri constants of some ample line bundles on it. We also examine the Seshadri constants of some ample line bundles on $X^4_{r,0}$ ($r\leq 7$) and $X^5_{r,0}$ ($r \leq 5$).
\end{abstract}

\section{introduction}
Let $X$ be a projective variety over an algebraically closed field $k$ of characteristic zero. For a nef line bundle $L$ on $X$, the \textit{Seshadri constant} of $L$ at a point $x \in X$ is defined as 
\begin{align*}
\epsilon(X, L, x) := \inf_{x \in C} \frac{L \cdot C}{\mult_x C},
\end{align*}
where the infimum is taken over all curves $C$ in $X$ passing through $x$. Here  $L \cdot C$ and $\mult_x C$ denote the intersection number and multiplicity of $C$ at $x$ respectively. Sometimes we write $\varepsilon(L, x)$ instead of $\varepsilon(X, L, x)$ if there is no confusion about the variety $X$.

Seshadri constant measures the local positivity of a nef line bundle on a projective variety around a given point. Motivated by the Seshadri's criterion for ampleness of a line bundle, the Seshadri constant was introduced by Demailly in 1992 mainly to tackle the Fujita conjecture. Over the years it has emerged as an interesting invariant and has grown on its own as an important topic of research in algebraic geometry.
Most of the research on Seshadri constants has been conducted in the case of surfaces. One of the most important results is due to Ein and Lazarsfeld (see \cite{EL}), who proved that for any ample line bundle $L$ on a surface $X$ the Seshadri constant $\varepsilon(X, L, x) \geq 1$ for a general point $x\in X$. Another fundamental result on Seshadri constant was shown by Miranda (see \cite{L1}, Example 5.2.1), which says that given a real number $\delta > 0$, there exists an algebraic surface $X$, an ample line bundle $L$ on $X$ and a point $x \in X$ such that $\varepsilon(X, L,x) < \delta$. This result in particular says that Seshadri constants can be arbitrarily small.

Studying Seshadri constants on higher-dimensional varieties has not been as successful as on surfaces. A few cases where the Seshadri constants have been studied on higher-dimensional varieties are Fano varieties (see \cite{Lee}), toric varieties (see \cite{I}), abelian varieties (see \cite{B}, \cite{L2}, \cite{N}), Grassmann bundles over curves (see \cite{BHNN}), and quot schemes (see \cite{GHS}) etc. For a detailed exposition about the work on Seshadri constants, see \cite{BDHKKSS}.

Computing Seshadri constants precisely is more often than not a very difficult task. So most of the work on Seshadri constants have been on giving bounds on them and sharpening the already known bounds when precise computation is difficult.

In this paper, we explore Seshadri constants on blow-ups of projective spaces $\mathbb{P}^n$ for $n=3,4$ and $5$. In Section \ref{preliminaries}, we establish the necessary notations and definitions and provide a brief overview of Seshadri constants. Moving on to Section \ref{points blow-up}, we focus on $X_{0,s}^n$, which is a blow-up of $\mathbb{P}^n$ at $s$ general points, and present our  main result, Theorem \ref{theorem-1}. This theorem demonstrates that for $s < 2n$ and for a \textit{l-very ample line bundle} $L$ (as defined in \cite{DP}), $\varepsilon(L, x) \geq l$. In Theorem \ref{theorem-1.1}, we extend this conclusion to the case where $s \geq 2n$, assuming some additional conditions. 

Let $X_{r,0}^n$ be the blow-up of $\mathbb{P}^n$ at $r$ general lines in $\mathbb{P}^n$. In Section \ref{section-Line blow-up}, using some results from \cite{DPU}, we compute the nef cones of line bundles on $X_{r,0}^3$ for $r \leq 6$, as seen in Theorem \ref{theorem-2}, and study Seshadri constants on them in Theorem \ref{Seshadri-constant-on-lineblowup-of-P3} and Proposition \ref{Prop:n=3;r=5;Sesh_cons}. In Section \ref{Line blow-ups of}, we draw upon results regarding the positive cones of $X_{r,0}^4$ for $r \leq 7$ and $X_{r,0}^5$ for $r\leq 5$ from \cite{PP} to establish conclusions about the Seshadri constants of nef line bundles, as presented in Theorem \ref{Seshadri-constant-on-lineblowup-of-P4}, Proposition \ref{prop about P4} and Theorem \ref{Sesh-cons-on-blow-up-of-P5}.

\section{Preliminaries}\label{preliminaries}
Let $\pi_{r,s} : X_{r, s}^n \rightarrow \mathbb{P}^n$ be the blow-up of $\mathbb{P}^n$ at $r$ general lines $l_1,..., l_r$ and $s$ general points $ p_1,..., p_s$. In this article, our primary focus will be on $X_{0,s}^n$, a blow-up of $\mathbb{P}^n$ at $s$ general points $p_1, ..., p_s$, and on $X_{r,0}^n$, the blow-up of $\mathbb{P}^n$ at $r$ general lines $l_1, ..., l_r$. We denote the respective blow-up maps as $\pi_{0,s} : X_{0,s}^n \rightarrow \mathbb{P}^n$ and $\pi_{r,0} : X_{r,0}^n \rightarrow \mathbb{P}^n$.

The real N\'eron-Severi groups, denoted as $N^1(X_{r, s}^n)$, are generated by $H$, $e_i, \, i=1,...,s$ and $E_i, \, i=1, ..., r$. Here, $H$ represents the pullback of the hyperplane class in $\mathbb{P}^n$, $e_i$ corresponds to the exceptional divisor associated to the point $p_i$ for $i=1,2,...,s$ and $E_i$ corresponds to the exceptional divisor associated to the line $l_i$ for $i=1,2,...,r$.

Let $X$ be a projective variety over any field $k$. A line bundle $L$ on $X$ (or a Cartier divisor $D$ with $\mathbb{Z}$ or $\mathbb{R}$ coefficients) is \textit{nef} if
\begin{align*}
\int_C c_1(L) \geq 0 \quad 
\left( \text{or} \; D \cdot C \geq 0 \right),
\end{align*}
for all irreducible curves $C$ in $X$.
The \textit{nef cone} $\Nef(X) \subset N^1(X)$ is the convex cone of all nef $\mathbb{R}$-divisor classes on $X$.
\begin{comment}
Let $X$ be a projective variety and $L$ be a line bundle on $X$. The Seshadri constants of $L$ at a point $x\in X$ if defined as
\begin{align*}
\varepsilon(X, L, x) := \inf _{x \in C} \hspace{1mm} \left\{ \frac{L \cdot C}{\mult_xC} \right\},
\end{align*}
where the infimum is taken over all irreducible curves $C\subset X$ passing through $X$.
\end{comment} 

Coming back to Seshadri constant, it is not difficult to see using the definition of Seshadri constant of an ample line bundle $L$ at a point $x\in X$, that $\varepsilon(X,L,x) \leq \sqrt[n]{L^n}$. Also, Seshadri's criterion for ampleness says that $L$ is ample if and only if $\varepsilon(X,L,x) >  0$ for all $x \in X$. Now, for an ample line bundle $L$ on $X$, we define
\begin{eqnarray*}
 \varepsilon(X,L,1) &:=&  \sup_{x\in X} \hspace{1mm} \{\varepsilon(X,L,x)\}, \ \text{and} \\
 \varepsilon(X,L) &: =& \inf_{x\in X} \hspace{1mm}\{\varepsilon(X,L,x)\}.
\end{eqnarray*}
Note that $0 < \varepsilon(X,L) \leq \varepsilon(X,L,x) \leq \varepsilon(X,L,1) \leq \sqrt[n]{L^n}$ holds for every point $x \in X$.

\section{Point blow-ups of $\mathbb{P}^n$}\label{points blow-up}
In this section, we prove some results on Seshadri constants of $l$-very ample line bundles on $X_{0,s}^n$. 
But before that we state the following definition which we use.
\begin{defi}
Let $X$ be a smooth complex projective variety. For an integer $l \geq 0$, a line bundle $\mathcal{O}_{X}(D)$ on $X$ is said to be $l$-very ample if for any 0-dimensional subscheme $Z \subset X$ of length $ h^0(Z, \mathcal{O}_Z) = l + 1$, the restriction map $H^0(X, \mathcal{O}_{X}(D)) \longrightarrow H^0(Z, \mathcal{O}_{X}(D) |_{Z})$ is surjective.
\end{defi}
\vspace{1mm}
Let $L = dH - \sum\limits_{i=1}^{s}m_ie_i$ be a line bundle on $X_{0,s}^n$. Fix $l\geq 0$, then for any $s \geq n+3$ we recall the following integer introduced in \cite[Eq. 2.1]{DP}:
$$
b_l := 
\begin{cases}
\min\{n -1, s-n-2\} - l -1, & \text{if $m_1 = d-l-1$ and $m_i = 1$}, \forall \ i\geq 2, \\ 
\min\{n, s-n-2\} - l - 1, & \text{otherwise}.
\end{cases}
$$
For $s \leq n+2$, define $b_l := -l-1$. 

We now recall the following characterisation of a line bundle $L$ on $X_{0,s}^n$ to be $l$-very ample. 

\begin{prop}[Theorem 2.2, \cite{DP}]\label{proposition}
Let $\pi_{0,s} : X_{0,s}^n \to \mathbb{P}^n$ be a blow-up of $\mathbb{P}^n$ at $p_1, p_2, ..., p_s \in \mathbb{P}^n$ and let the points $p_1, p_2, ..., p_s$ be in general positions. Let $L = dH - \sum\limits_{i=1}^{i=s}m_ie_i$ be a line bundle. Assume that either $s \leq 2n$ or $s\geq 2n+1$ and $d$ is large enough namely
\begin{align*}
\sum_{i=1}^s m_i - nd \leq b_l,
\end{align*}
where $b_l$ defined as before. Then for any non zero integer $l$, a line bundle $L$ is $l$-very ample if and only if
\begin{align*}
m_i \geq l, \quad \forall \; i \in \{1, 2, ..., s\}, \; \text{and} \\
d-m_i - m_j \geq l, \quad \forall \; i, j \in \{1, 2, ..., s\} \, (i \neq j).
\end{align*}
\end{prop}

Now, we will proceed to prove the first main theorem of this section. 
\begin{thm}\label{theorem-1}
Let $\pi_{0,s} : X_{0,s}^n \to \mathbb{P}^n$ be a blow-up of $\mathbb{P}^n$ at $s\leq 2n$ points $p_1, p_2, ..., p_s$ in general positions. Let $L = dH - \sum_{i=1}^{i=s}m_ie_i$ be a $l$-very ample line bundle on $X_{0,s}^n$, and let $x$ be a point in $X_{0,s}^n$. Then, we have
\begin{align*}
\varepsilon(L, x) \geq l.
\end{align*}
\end{thm}
\begin{proof}

Let $B$ be a reduced irreducible curve in $X_{0,s}^n$ passing through $x$ and $m := \mult_x B$. Then either $\pi_{0,s}(B) = \text{point}$ or $B$ is the strict transform of a curve $C$ in $\mathbb{P}^n$ i.e., $B= \tilde{C}$. \\
\underline{Case I:} \, Let $1 \leq s \leq n$.

Let $\pi_{0,s}(B) = \textit{point}$ i.e., $\pi_{0,s}(B) \in \{p_1, p_2, ..., p_s \}$, then 
$B \subset e_i$ for some $i$ and therefore it is numerically equivalent to 
$\text{\text{deg(B)}}f_i$ where $f_i$ denotes the class of a line in $e_i$. 
Note that the intersection table is given by 
\[
H \cdot f_i = 0, \quad \text{and} \quad e_j\cdot f_i = -\delta_{i,j}.
\]
Since $B$ passes through a point $x$ of multiplicity $m$, we have $\text{deg(B)} \geq m$. Therefore,
\begin{align*}
\frac{L \cdot B}{m} = \frac{(\deg B)m_i}{m} \geq m_i \geq l
\end{align*}
where that last inequality follows from Proposition \ref{proposition}. 

Now, assume that $B$ is the strict transform of a curve $C$ in $\mathbb{P}^n$ i.e., $B= \tilde{C}$. Then
\begin{eqnarray}\label{equation:1}\nonumber
\frac{L \cdot B}{m} &=& \frac{(dH - \sum_{i = 1}^{s} m_i e_i) \cdot B}{m} \\ \nonumber
&\geq&  \frac{d\cdot \text{deg(C)} - \sum_{i=1}^{s}m_i \mult_{p_i}C}{m} \\ 
&\geq&   \frac{l\cdot \text{deg(C)}}{m} + \frac{(d-l)\cdot \text{deg(C)} - \sum_{i=1}^{s}m_i \mult_{p_i}C}{m}.
\end{eqnarray}
Now, we re-arrange $p_1, p_2, ...,p_s$, if necessary, such that $\mult_{p_1}C \geq \mult_{p_2}C \geq ... \geq \mult_{p_s}C$. Choose the largest integer $k$ such that $\mult_{p_k}C \neq 0$. Since  $p_1, p_2, ..., p_s$ are in general positions 
and $s\leq n$, we can find a hyperplane $\bar{H}$ which passes through $p_1,..., p_{k-1}, p_{k+1},..., p_{s}$ but misses $p_k$. Therefore, $C \nsubseteq \bar{H}$, and by B\'{e}zout's theorem, we obtain:
\begin{eqnarray*}
\text{deg(C)} = \bar{H} \cdot C \geq \sum_{i=1}^{k-1} \mult_{p_i}C.  
\end{eqnarray*}
So,
\begin{eqnarray}\label{equation:2}
(d-l) \bar{H}\cdot C &\geq& (d-l)\sum_{i=1}^{k-1} \mult_{p_i}C \cr 
&\geq& \sum_{i=1}^{k-2} m_i \mult_{p_i}C + (m_{k-1} + m_k) \mult_{p_{k-1}}C \cr 
&\geq& \sum_{i=1}^{k-2} m_i \mult_{p_i}C + m_{k-1} \mult_{p_{k-1}}C + m_k \mult_{p_k}C \cr
&\geq&  \sum_{i=1}^{k} m_i \mult_{p_i}C \cr 
&=& \sum_{i=1}^{s} m_i \mult_{p_i}C
\end{eqnarray}
since by Proposition \ref{proposition}, $(d-l) \geq m_i + m_j$  for $i \neq j$ and in particular, $d-l \geq m_i$ for all $i$.\\
Then continuing the calculation from \eqref{equation:1} and using \eqref{equation:2}, the Seshadri ratio becomes
\begin{align*}
\frac{L\cdot B}{m} \geq \frac{l\cdot \text{deg(C)}}{m} \geq l.
\end{align*}
\vspace{2mm} 
\underline{Case II:}  \,   Assume $n \leq s \leq 2n$. \\
If $\pi_{0, s}(B) = point $, then we proceed as in Case I and obtain $\frac{L\cdot B}{m} \geq  l$.
\vspace{2mm}
Otherwise, $B$ is the strict transform of a curve $C$ in $\mathbb{P}^n$ i.e., $B= \tilde{C}$. 
As in Case $1$, we rearrange the points such that $\mult_{p_1}C \geq \mult_{p_2} C \geq... \geq \mult_{p_s} C$.
Since $p_1, p_2, ..., p_n$ are in general position, they span a hyperplane $ H_1 = \langle p_1, p_2, ..., p_n \rangle$. 
Now, if $\mult_{p_i}C \neq 0$ for any $n+1\leq i \leq s$, then $C \not\subset H_1$ and by B\'ezout's theorem we obtain
\begin{align*}
H_1 \cdot C \geq \sum_{i = 1}^n \mult_{p_i} C. 
\end{align*}
From this we get (after substituting $m_i = 0$ for $i=s+1,s+2,...,2n$),
\begin{eqnarray*}
(d -l) H_1 \cdot C   &\geq& (d - l) \sum_{i = 1}^n \mult_{p_i}C \cr 
&\geq&  \sum_{i = 1}^ n (m_i + m_{i +n})\mult_{p_i}C \cr
&\geq& \sum_{i=1}^n m_i (\mult_{p_i}C) + \sum_{i + n \leq s} m_{i +n} (\mult_{p_{i+ n}} C) \cr
&\geq& \sum_{i =1}^s m_i (\mult_{p_i}C), 
\end{eqnarray*}
which implies that
\begin{eqnarray}\label{Iniquality to be used in next thm}
\frac{(d -l) H_1 \cdot C - \sum_{i =1}^s m_i (\mult_{p_i}C)}{m} \geq 0.
\end{eqnarray}
Therefore,
 \begin{eqnarray*}
\frac{L \cdot B}{m} &=& \frac{(dH - \sum_{i = 1}^{s} m_i e_i) \cdot B}{m} \cr 
&\geq&  \frac{dH_1\cdot C - \sum_{i=1}^s m_i \mult_{p_i}C}{m} \cr
&\geq& \frac{l H_1\cdot C}{m} + \frac{(d-l)H_1\cdot C - \sum_{i=1}^s m_i \mult_{p_i}C}{m} \cr
&\geq& \frac{l H_1\cdot C}{m} \cr
&\geq& l.
\end{eqnarray*}
Now, if $\mult_{p_i}C = 0$ for all $i \geq n+1$, and $C \not\subset H_1$ we go back to Case I. Therefore, assume that $\mult_{p_i}C = 0$ for all $i \geq n+1$ and that $C \subset H_1$. In this case, we restrict 
our attention to the following setting:
Let
\[
\pi_{0,s}|_{H_{1_{0,n}}} : H_{1_{0,n}} \hookrightarrow  X_{0,s}^n \to H_1 \cong \mathbb{P}^{n-1} \hookrightarrow \mathbb{P}^n,
\]
be a blow-up of $H_1$ at $p_1,...,p_n$, with $L|_{H_{1_{0,n}}} = dH_2 - \sum\limits_{i=1}^{i=n}m_ie_{i,2}$. Here, $H_2$ and $e_{i,2}$ 
denotes the pullback of a hyperplane in $\mathbb{P}^{n-1}$ and the exceptional divisors 
corresponding to the map $\pi_{0,s}|_{H_{1_{0,n}}}$ respectively. Note that $e_{i} \cap {H_{1_{0,n}}} = e_{i,2}$. 
Under these notations we have 
\begin{eqnarray}\label{equation:3}
\frac{L\cdot B}{m} = \frac{L|_{H_{1_{0,n}}} \cdot B}{m} = 
\frac{(dH_2 -  \sum\limits_{i=1}^{n}m_ie_{i,2})\cdot B}{m} = \frac{d\cdot \text{deg(C)} - \sum\limits_{i=1}^{n}m_i\mult_{p_i}C}{m}.
\end{eqnarray}
Now, proceeding as in the Case I, we choose a hyperplane $H_2'$ which contains $p_1,...,p_{k-1},p_{k+1},...,$ $p_n$ 
but misses $p_k$, where $k$ is the largest number such that $\mult_{p_k}C \neq 0$. Then $C\not\subset H_2'$ and by 
B\'ezout's theorem we get
\[
H_2' \cdot C \geq \sum\limits_{i=1}^{k-1}\mult_{p_i}C.
\]
Now, following \eqref{equation:1},\eqref{equation:2} and \eqref{equation:3} we get the required result.
Therefore, from Case I and Case II we conclude that $\varepsilon(L, x) \geq l$. 
\end{proof}
 When $2n < s \leq 3n -1$, we obtain the following result concerning the Seshadri constant of an $l$-very ample line bundle of semi-homogeneous type.
 \begin{thm}\label{theorem-1.1}
 Let $\pi_{0,s} : X_{0,s}^n \to \mathbb{P}^n$ be a blow-up of $\mathbb{P}^n$ at points $p_1, p_2, ..., p_s$ in general positions, where $ 2n < s\leq 3n -1$. Let $L = dH - \sum_{i=1}^{s}m_ie_i$, be an $l$-very ample line bundle on $X_{0,s}^n$ with $d$ large enough, namely $\sum_{i=1}^s m_i - nd \leq b_l$, where $b_l$ is as defined in Proposition \eqref{proposition}. Let $m_i = l$ for $i = 2n +1, ..., 3n - 1$, and $x \in X_{0,s}^n$ be any point. Then, we have
 \begin{align*}
 \varepsilon(L, x) \geq l.
 \end{align*}
   \end{thm}
   
\begin{proof}
Let $B$ be a reduced and irreducible curve in $X_{0,s}^n$ passing through the point $x$ with multiplicity $m$. If $\pi_{0,s}(B) = point$, then as we have shown in Theorem \ref{theorem-1}, $\frac{L\cdot B}{m} \geq l$. So, let's assume that $B$ is strict transform of a curve $C$ i.e., $B = \tilde{C}$. Also, rearrange the points as in the previous theorem, if necessary, such that $\mult_{p_1}C \geq \mult_{p_2} C \geq... \geq \mult_{p_s} C$. Moreover, assume that $\mult_{p_{2n+1}}C \neq 0$; otherwise, we reduce to the previous theorem. Then
   \begin{eqnarray*}
   \frac{L \cdot B}{m} &=& \frac{\left(dH - \sum_{i = 1}^{s} m_i e_i \right) \cdot B}{m} \\
    &=& \frac{(l H - \sum_{i = 2n +1}^{3n -1}m_i e_i) \cdot B + ( (d - l)H -\sum_{i = 1}^{2n} m_i e_i)\cdot B}{m}.
   \end{eqnarray*}
We can then choose a hyperplane $H'$ passing through $p_{2n+1}, ..., p_{3n-1}$ and $\pi(x)$ that doesn't contain $p_{2n}$. Then $C \not\subset H'$ and therefore by B\'{e}zout's theorem we get
   \begin{align*}
   H' \cdot C \geq \sum\limits_{i = 2n+1}^{3n -1} \mult_{p_i} C + m.
   \end{align*}
Now, the fact that $m_i = l$ for $i = 2n +1, ..., 3n - 1$ further implies that
\[
lH\cdot B = lH' \cdot C \geq  \sum\limits_{i = 2n+1}^{3n -1}l (\mult_{p_i} C) + lm \geq \sum\limits_{i = 2n+1}^{3n -1}m_i (\mult_{p_i} C) + lm.
\] 
Therefore, 
\[
\frac{(l H - \sum_{i = 2n +1}^{3n -1}m_i e_i) \cdot B}{m} \geq l.
\]
Also, from the inequality \eqref{Iniquality to be used in next thm} of proof of Theorem \ref{theorem-1}, we see that  $\left((d - l)H -\sum\limits_{i = 1}^{2n} m_i e_i \right)\cdot B \geq 0$.
Hence the Seshadri ratio $\frac{L \cdot B}{m} \geq l$ implies $\varepsilon(L, x) \geq l$.        
    \end{proof}

\section{Line blow-ups of $\mathbb{P}^3$}\label{section-Line blow-up}
We now turn our attention to $X_{r,0}^3$, i.e., blow-up of $\mathbb{P}^3$ at $r$ general lines. First we compute the nef cone of divisors on $X_{r,0}^3$ for $r\leq 6$, and then study the Seshadri constants of nef line bundles on $X_{r,0}^3$. Before stating our next theorem, we would like to note the following: \\
Given any three, among the four general lines $l_1,...,l_4$ in $\mathbb{P}^3$, there exists a unique quadric $\mathcal{Q}_{i,j,k}$ (which is isomorphic to $\mathbb{P}^1 \times \mathbb{P}^1$) containing three of them on one ruling \cite[Lemma 3.1]{DPU}. The remaining fourth line $l_m$ ($m \neq i,j,k$) intersect $\mathcal{Q}_{i,j,k}$ at two points through which there exist two lines $t,t'$ lying on the other ruling and intersecting all four given lines transversely. It is not difficult to see that, 
\[
\mathcal{Q}_{i,j,k} \cap \mathcal{Q}_{i,j,m} = l_i\cup l_j\cup t\cup t'
\]
for all (distinct) $i,j,k,m \in \{1,2,3,4\}$. Therefore, 
\[
\bigcap\limits_{i\neq j\neq k \in \{1,..,4\}} \mathcal{Q}_{i,j,k} = t\cup t'.
\]
Now if we have $l_1,...,l_5 \subset \mathbb{P}^3$ to be the five general lines, then for every four $l_1,...,\hat{l_i},...,l_5$ of them we get two transversal lines which we denote by $t_i,t_i'$.
\begin{thm}\label{theorem-2}
 Let $X_{r,0}^3$ be a blow-up of $\mathbb{P}^3$ at $r$ general lines $l_1, l_2, ... , l_r$. Then for $r \leq 6$, the nef cones of divisors on $X_{r,0}^3$ are described as:
\begin{enumerate} 
\item $\Nef( X_{r,0}^3) = \big\{dH - \sum\limits_{i=1}^r m_iE_i  \, \,|  \, \, d\geq \sum\limits_{i=1}^r m_i,  \, \, d \geq 0, m_i \geq 0 \; \forall \; i \in \{1,...,r\} \big\}$ for $r=1, ..., 4$.
 
\item $\Nef(X_{5,0}^3) = \big\{ dH - \sum\limits_{i=1}^5 m_iE_i  \, \,|  \, \, d\geq (m_1+ m_2 + ... + m_5) - m_i, \;  d \geq 0, m_i \geq 0 \; \forall \; i \in \{1,..., 5\} \big\}. $
 
\item $\Nef(X_{6, 0} ^3) = \big\{ dH - \sum\limits_{i=1}^6 m_iE_i  \, \,|  \, \, d \geq m_i +m_j +m_k + m_l, \, \,  i <  j <  k <  l , \{ i, j, k, l\} \in \{1, 2, ..., 6\}, \, \, d \geq 0, \, \,m_i \geq 0 \; \forall \; i \in \{1,..., 6\} \big\}.$
 \end{enumerate}
\end{thm}
\begin{proof}
Let $\tilde{l_i}$ be a class of a line inside the exceptional divisor $E_i$, for $i= 1,..., 6$. If $D = dH - \sum_i^r m_i E_i$ is a nef divisor on $X_{r,0}^3$, then intersecting $D$ with $l$ (a class of general line in $X_{r,0}^3$) and $\tilde{l_i}$ for $i= 1,..., 6$ 
gives $d \geq 0$ and $ m_i \geq 0$ for $i=1,..., 6$.\\
$\bullet$ $\underline{r=1}$:
Let $D= dH - m_1E_1$ be a nef divisor on $X_{1,0}^3$. Consider the line $\tilde{l} - \tilde{l_1}$, then $(dH - m_1E_1) \cdot (\tilde{l} - \tilde{l_1}) = d- m_1 \geq 0$. On the other hand $H$ and $H-E_1$ are nef line bundles on $X_{1,0}^3$ and any divisor $D$ on $X_{1,0}^3$ can be written as $D= dH -m_1E_1 = (d- m_1)H + m_1(H - E_1)$.\\
$\bullet \; \underline {r=2}$:
Let $D= dH - \sum\limits_{i=1}^2m_iE_i$ be a nef divisor on $X_{2,0}^3$  and let us consider the line $\tilde{l}- \tilde{l_1} - \tilde{l_2}$. Since $D$ is nef, $\left(dH- \sum\limits_{i=1}^2 m_iE_i \right) \cdot (\tilde{l}- \tilde{l_1} -\tilde{l_2}) = d- m_1 -m_2 \geq 0$. Conversely, any divisor $D$ on $X_{2,0}^3$ can be written as $D= dH -m_1E_1-m_2E_2 = (d- m_1 - m_2)H + m_1(H - E_1) + m_2(H -E_2)$.\\
$\bullet \; \underline{r=3}$:
Let $D= dH - \sum\limits_{i=1}^3m_iE_i$ be a nef divisor on $X_{3,0}^3$. 
       By Theorem 3.1 in \cite{DPU}, for any three general lines $l_i, l_j, l_k$ in $\mathbb{P}^3$ there exists a unique smooth quadric surface $Q_{i,j,k}$ containing the lines $l_i, l_j, l_k$. If we take our lines as $l_1, l_2, l_3$ in this case then by generality all the three lines should belong to the same ruling of $Q_{1,2,3}$. Then for any line $l_0$ in the other ruling of $Q_{1,2,3}$, $\left(dH - \sum\limits_{i=1}^3m_1E_1  \right) \cdot l_0  = d- m_1 -m_2 - m_3 \geq 0$. Conversely, we write $dH -\sum\limits_{i=1}^3 m_iE_i = (d - m_1 -m_2 -m_3)H + \sum\limits_{i=1}^3 m_i(H - E_i)$.\\
$\bullet \; \underline{r=4}$: 
By Lemma 3.2 in \cite{DPU}, any four general lines $l_i, l_j, l_k, l_l$ in $\mathbb{P}^3$ determine two transversal lines $t, t'$ 
and $D \cdot t = \left(dH- \sum\limits_{i=1}^4 m_i E_i \right) \cdot t = d -m_1 -m_2- m_3  - m_4$. Since $D$ is a nef divisor on $X_{4,0}^3$, we have $d -m_1 -m_2- m_3  - m_4 \geq 0$.
Conversely, we write $dH -\sum\limits_{i=1}^4 m_iE_i = (d - m_1 -m_2 -m_3 - m_4)H + \sum\limits_{i=1}^4 m_i(H - E_i)$.\\
$\bullet \; \underline{r=5}$: 
As mentioned in the previous case, for each choice of four lines ${l_1,..., \hat{l_i},... l_5}$ there exist two transversal lines ${t_i, {t_i}'}$. So for any nef divisor $D$ on $X_{5,0}^3$, $D \cdot t_i = \left(dH - \sum\limits_{i=1}^5 m_i E_i \right) \cdot t_i = d - \{(m_1 + ... +m_5) - m_i\} \geq 0$ $\forall$ $i = 1, 2, ..., 5$.
    
    Conversely, w.l.o.g. assume that $m_1 = \min \{m_1, ..., m_5 \}$ after re-arranging the lines if needed. Write any divisor $D= dH - \sum\limits_{i=1}^5 m_i E_i$ as $dH - \sum\limits_{i=1}^5 m_i E_i = m_1(4H - \sum\limits_{i=1}^5 E_i) + (d- 4m_1) H - \sum\limits_{i=1}^5 (m_ i - m_1)E_i$. Now, by Theorem 4.4 in \cite{DPU}  the anti-cannonical divisor $ -\tilde{K_3} = 4H - \sum\limits_{i=1}^5 E_i$ is nef. The other part is $ (d - 4m_1)H - \sum_{i=2}^5(m_i - m_1)E_i$, which can be considered as a divisor on $X_{4,0}^3$ (considered as the blow-up of $\mathbb{P}^3$ at $l_2, ..., l_5$). Since $m_1 = \min \{m_1,..., m_5\}$ and $d \geq (m_1 +...+ m_5) - m_i$ for all $i \in \{1,2,..., 5\}$, we have $d -4m_1 \geq 0$. So by $r=4$ case, $(d - 4m_1)H - \sum\limits_{i=2}^5(m_i - m_1)E_i$ is nef on $X_{4,0}^3$. Therefore, $(d - 4m_1)H - \sum\limits_{i=2}^5(m_i - m_1)E_i$ is nef on $X_{5,0}^3$, since $X_{5,0}^3$ is obtained by blowing up $l_1$ in $X_{4,0}^3$. As a result we conclude that $D= dH - \sum\limits_{i=1}^5 m_i E_i$ is nef.\\
$\bullet \; \underline{r=6}$: 
Using similar argument as in the previous case, that is, by choosing 4 lines at a time, we can say that a nef divisor $D=dH - \sum\limits_{i=1}^6 m_i E_i$ satisfies the conditions $d \geq m_i +m_j + m_k + m_l$,   $(i<j<k<l)$, $\{ i, j, k, l\} \in \{1, 2, ..., 6\} $. Conversely, any divisor $D$ can be written as $D = dH - \sum\limits_{i=1}^6 m_i E_i = m_1(4H - \sum\limits_{i=1}^6 E_i) + (d- 4m_1) H - \sum\limits_{i=2}^6 (m_ i - m_1)E_i$. Now $(4H - \sum\limits_{i=1}^6 E_i)$ is the anti-canonical divisor, hence nef by \cite[Theorem 4.4]{DPU}. The fact that the other part  i.e., $(d- 4m_1) H - \sum\limits_{i=2}^6 (m_ i - m_1)E_i$ is nef can be shown by the hypothesis and $r=5$ case. So, $D$ is nef on $X_{6,0}^3$.            
\end{proof}
Next we explore the Seshadri constant of an ample line bundle at a general point on $X_{r,0}^3$.
In the Theorem \ref{Seshadri-constant-on-lineblowup-of-P3} (respectively, Theorem \ref{Seshadri-constant-on-lineblowup-of-P4}) below
both the class of a general hyperplane in $\mathbb{P}^3$  (respectively, $\mathbb{P}^4$) and its pull back in 
$X_{r,0}^3$ (respectively, $X_{r,0}^4$) are denoted by $H$. Let $L = dH - \sum\limits_{i=1}^r m_i E_i$ 
be an ample line bundle on $X_{r,0}^k$, the blow of $\mathbb{P}^k$ ($k=3,4,5$) at $r$ general lines. We then denote
\begin{eqnarray*}
 m_L := \min\{m_1,m_2,...,m_r\}.
\end{eqnarray*}
For a curve $C$ in $\mathbb{P}^k$ ($k=3,4,5$), we also denote by $C\cdot l_i = n_i$ to be the length 
of zero dimensional scheme $\{x_{j}^{i} \ |\ 1\leq j \leq k_i,\ \text{for some $k_i \geq 0$} \}$ of intersection of $C$ and $l_i$, and call $n_i = \sum\limits_j n_j^i$ to be the intersection multiplicity of $C$ along $l_i$. Here $n_{j}^{i} := \mult_{x_{j}^{i}}C$ is the usual multiplicity of $C$ at $x_{j}^{i}$. 

\begin{thm}\label{Seshadri-constant-on-lineblowup-of-P3}
Let $X_{r,0}^3$ denote a blow-up of $\mathbb{P}^3$ along $r \leq 4$ general lines $l_1, l_2,...,l_r$ and $x$ be a general point of $X_{r,0}^3$. Let $L = dH- \sum\limits_{i=1}^{r}m_iE_i$ be an ample line bundle on $X_{r,0}^3$, 
then
\begin{enumerate}
    \item $\varepsilon(X_{r,0}^3, L, 1)= d - \sum\limits_{i=1}^{r}m_i,$ \, for $r = 1,2$ 
    \item $\varepsilon(X_{3,0}^3, L, x) \geq d-m_1-m_2-m_3$, with equality if $\pi_{3,0}(x) \in \mathcal{Q}_{1,2,3}$ 
    \item $\varepsilon(X_{4,0}^3, L, x) \geq d-m_1-m_2-m_3-m_4$, with equality if $\pi_{3,0}(x) \in t\cup t'$ 
    \item $\varepsilon(X_{4,0}^3, L, x) \in [d-m_1-m_2-m_3-m_4, \ d-m_i-m_j-m_k]$, if $\pi_{3,0}(x) \in \mathcal{Q}_{i,j,k}\setminus(t \cup t')$.
\end{enumerate}
\end{thm}

\begin{proof}
Let $x \in X_{r,0}^3$ be a general point and let $B$ be a reduced and irreducible curve in $X_{r,0}^3$ passing through $x$ with multiplicity $m$. Since $x$ is assumed to be a general point, we can assume that $x$ does not lie on any exceptional divisors. So, there is a curve $C$ in $\mathbb{P}^3$ such that $B$ is the strict transform of $C$ i.e., $B = \tilde{C}$. Assume that $C$ has multiplicities $n_i$ along $l_i$ and $x_j^i$ ($1\leq j \leq k_i$) are points such that $\sum\limits_{j=1}^{k_i}\mult_{x_j^i}C = \mult _{l_i}C$ for  $i = 1,..., r$. Let $\tilde{l}$ be the pullback of a general line in $\mathbb{P}^3$ and $\tilde{l_i}$ be the class of a line inside $E_i$ for $i=1,..., r$. In each of the following cases, after re-indexing the points $x_1^i,x_2^i, ..., x_{k_i}^i$ if necessary, we assume that $ n_j^i \geq n_{j+1}^i$ for $j = 1,..., k_i-1$. We also re-index the lines $l_1,...,l_r$ if necessary, to assume that $n_i \geq n_{i+1}$ for $i=1,2,...,r-1$.\\
$\bullet$ $\underline {r = 1}$:
Consider the line connecting $\pi_{1,0}(x)$ and $x_1^1$ whose strict transform is $\tilde{l} - \tilde{l_1}$. If $\tilde{C} = \tilde{l} - \tilde{l_1}$ then the Seshadri ratio is $d - m_1$. If $\tilde{C} \neq \tilde{l} - \tilde{l_1}$, then there is a hyperplane $H$ containing $x_1^1$ and $\pi_{1,0}(x)$ but not containing $C$. So, by B\'{e}zout's theorem we get
\begin{eqnarray}\label{eqn:1}
H \cdot C \geq n_1^1 + m,
\end{eqnarray}
Here we consider the following two cases as follows: \\
\textit{Case 1:} ($m \leq C\cdot l_1 - 1$) Note that in this case 
\begin{eqnarray*}
    C\cdot l_1 - 1 \geq m \\
    \Rightarrow \frac{C\cdot l_1}{m} \geq 1+\frac{1}{m}. 
\end{eqnarray*}

Since $H\cdot C \geq C\cdot l_1$, we get 
\begin{eqnarray*}
    \frac{L\cdot B}{m} &=& \frac{dH\cdot C-m_1C\cdot l_1}{m} \cr 
    &\geq& (d-m_1)(1+\frac{1}{m}) \cr 
    &>& d-m_1.
\end{eqnarray*}

\textit{Case 2:} ($m\geq C\cdot l_1$) In this case, since $H\cdot C \geq n_1^1 + m$, we have the following,
\begin{eqnarray*}
    \frac{dH\cdot C-m_1C\cdot l_1}{m} &\geq& \frac{d(n_1^1 + m) -m_1n_1}{m} \cr 
    &=& d + \frac{dn_1^1 - m_1n_1}{m}.
\end{eqnarray*}

Now note that 
\begin{eqnarray}\nonumber
\frac{dn_1^1 - m_1n_1}{m} \geq -m_1 &\Leftrightarrow& dn_1^1 - m_1n_1 + mm_1 \geq 0 \\ \nonumber
&\Leftrightarrow& dn_1^1 + (m-n_1)m_1 \geq 0, \nonumber
\end{eqnarray}
which holds, since $m \geq n_1$ by hypothesis.
So the Seshadri ratio is 
\begin{eqnarray*}
    \frac{L \cdot B}{m} &=& \frac{(dH - m_1E_1) \cdot B}{m} \cr 
    &=& \frac {dH\cdot C - m_1n_1}{m} \cr 
    &\geq& d - m_1.
\end{eqnarray*}

Therefore, 
\[
\varepsilon(X_{1,0}^3, L, 1)= d - m_1.\\
\]
$\bullet$ $\underline {r = 2}$:
As in $r=1$ case, we consider two cases here as well. \\ 
\textit{Case 1:} ($m \leq C\cdot l_1 - 1$) In this case, since 
$C\cdot l_1 \geq C\cdot l_2$ we have $-C\cdot l_2 \geq -C\cdot l_1$ and 
\begin{eqnarray}\nonumber
\frac{L\cdot B}{m} &=& \frac{dH\cdot C-m_1C\cdot l_1-m_2C\cdot l_2}{m} \\ \nonumber
&\geq& (d-m_1-m_2)\left(1+\frac{1}{m}\right),
\end{eqnarray}
since $H\cdot C \geq C \cdot l_1$. \\
\textit{Case 2:} ($m\geq C\cdot l_1$) 
We use \eqref{eqn:1} in this case to get
\begin{eqnarray}\nonumber
\frac{dH\cdot C-m_1C\cdot l_1-m_2C\cdot l_2}{m} &\geq& \frac{d(n_1^1 + m)-m_1n_1-m_2n_2}{m} \\ \nonumber
&=& d + \frac{dn_1^1 - m_1n_1-m_2n_2}{m}.
\end{eqnarray}
But, we have
\begin{eqnarray}\nonumber
\frac{dn_1^1 - m_1n_1-m_2n_2}{m} \geq -m_1-m_2 \\ \nonumber
\Leftrightarrow dn_1^1 + (m-n_1)m_1 + (m-n_2)m_2 \geq 0. 
\end{eqnarray}
This holds because we made the assumption that $m \geq n_1 \geq n_2$. \\
Therefore, we get 
\[
\frac{L\cdot B}{m} \geq d-m_1-m_2.
\]
Now, we show that this least value i.e., $d-m_1-m_2$ is attained by a line containing $\pi_{2,0}(x)$. Note that if we show the existence of a line $l_x$ containing $\pi_{2,0}(x)$ and intersecting $l_1$, $l_2$, we will be done, since the Seshadri ratio of $\tilde{l_x} = \tilde{l}-\tilde{l_1}-\tilde{l_2}$ is $d-m_1-m_2$ and hence
\[
\varepsilon(X_{2,0}^3, L, 1)= d-m_1-m_2.
\]
To see the existence of such line, we consider $J(l_1,l_2)$, the join of $l_1$ and $l_2$. Since these lines are disjoint we know that the dimension of subvariety $J(l_1,l_2)$ is $3$ \cite[Proposition 11.37]{JH}, and hence $J(l_1,l_2) = \mathbb{P}^3$. Therefore $\pi_{2,0}(x) \in \mathbb{P}^3 = J(l_1,l_2)$ and the existence of required line is proved.\\
$\bullet$ $\underline{r=3}$:
Using a similar argument as in the cases of $r=1$ and $r=2$ lines, we can show that %with required modification, that
 \begin{eqnarray}\nonumber
 \frac{L\cdot B}{m} &=& \frac{dH\cdot C-m_1C\cdot l_1-m_2C\cdot l_2-m_3C\cdot l_3}{m} \\ \nonumber
 &\geq& d-m_1-m_2-m_3.
 \end{eqnarray}
This implies that $\varepsilon(X_{3,0}^3, L, 1) \geq d-m_1-m_2-m_3$. Now, we know that 
given three lines $l_1,l_2,l_3$ in general position, there exists a quadric $\mathcal{Q}_{1,2,3} \cong \mathbb{P}^1 \times \mathbb{P}^1$ containing the lines $l_1,l_2,l_3$ on the same ruling. Therefore, if 
$\pi_{3,0}(x) \in \mathcal{Q}_{1,2,3}$, then we consider the line $l_x$ containing $\pi_{3,0}(x)$ and intersecting 
$l_1,l_2,l_3$ transversely. It is then clear that $\tilde{l_x} = \tilde{l} - \tilde{l_1} - \tilde{l_2}- \tilde{l_3}$ and its Seshadri ratio is $d-m_1-m_2-m_3$. Therefore,
\[
\varepsilon(X_{3,0}^3, L, x) = d-m_1-m_2-m_3
\]
if $\pi_{3,0}(x) \in \mathcal{Q}_{1,2,3}$. \\
$\bullet$ $\underline{r=4}$:
In this case also we have 
\begin{eqnarray}\nonumber
 \frac{L\cdot B}{m} &=& \frac{dH\cdot C-m_1C\cdot l_1-m_2C\cdot l_2-m_3C\cdot l_3-m_4C\cdot l_4}{m} \\ \nonumber
 &\geq& d-m_1-m_2-m_3-m_4,
\end{eqnarray}
for similar reason as in the case of $r=1,2$ lines. Now we know that given $4$ general lines $l_1,\cdots, l_4$ in $\mathbb{P}^3$ there exist two lines $t$ and $t'$ transversal to each $l_i$. 

Therefore if $\pi_{4,0}(x) \in (t \cup t')$, we get 
\[
\varepsilon(X_{4,0}^3, L, x) = d-m_1-m_2-m_3-m_4,
\]
and if $\pi_{4,0}(x) \in \mathcal{Q}_{i,j,k} \setminus(t \cup t')$, then
\[
\varepsilon(X_{4,0}^3, L, x) \in [d-m_1-m_2-m_3-m_4,\ d-m_i-m_j-m_k].
\]
\end{proof}
\begin{xrem}
Under the same notation as of previous theorem, note that if $x \in E_j$ for some $j=1,2,...,r$, and $B \subset E_j$ then 
\[
   \frac{L\cdot B}{m} = \frac{m_j\text{deg(B)}}{m} \geq m_j \ (\geq m_L).
\]
Also this value is obtained by a line in $E_j$ containing the point $x$.
Therefore, using this fact and the proof of previous theorem, it is clear that 
\[
\varepsilon(X_{r,0}^3, L) = \min\left\{m_L, d-\sum\limits_{i=1}^{r} m_i \right\}
\]
for $r=1,2,3,$ and $4$.  
\end{xrem} 
In the case of $r=5$, we present the following results. Similar results can also be observed for the 
$r=6$ case, which we omit to avoid repetition.
\begin{prop}\label{Prop:n=3;r=5;Sesh_cons}
Let $X_{5,0}^3$ denote a blow-up of $\mathbb{P}^3$ along $5$ general lines $l_1, l_2,...,l_5$ and $x$ be a general point of $X_{5,0}^3$. 
Let $L = dH- \sum\limits_{i=1}^{5}m_iE_i$ be an ample line bundle on $X_{5,0}^3$ such that $d \geq \sum\limits_{i=1}^{5}m_i$. Then 
\begin{enumerate}
    \item $\varepsilon(X_{5,0}^3, L, 1) \geq d -\sum\limits_{i=1}^{5}m_i$
    \item $\varepsilon(X_{5,0}^3, L, x) \in [d -\sum\limits_{i=1}^{5}m_i,\ \max\limits_{j}\{d -\sum\limits_{i\neq j}m_i \}]$, \text{if\ $\pi_{5,0}(x) \in \bigcup\limits_{m}(t_m \cup t_m')$}
    \item $\varepsilon(X_{5,0}^3, L, x) \in [d -\sum\limits_{i=1}^{5}m_i,\ \max\limits_{i,j,k}\{d -m_i-m_j-m_k\}], \text{if}\ \pi_{5,0}(x) \in \bigcup\limits_{i,j,k,m}(\mathcal{Q}_{i,j,k} \setminus(t_m\ \cup \  t_m'))$.
\end{enumerate}
\end{prop}
\begin{proof}
The proof for $(1)$ is clear. For $(2)$, note that if $\pi_{5,0}(x)$ is on $t_i$ or $t_i'$ for some $i \in\{1,..., 5\}$ and $C$ is either of the $t_i$ or $t_i'$, then the Seshadri ratio of $C$ is $d - \{ (\sum\limits_{j=1}^5 m_j) - m_i \}$. Therefore
\[
\varepsilon(X_{5,0}^3, L, x) \in [d -\sum\limits_{i=1}^{5}m_i,\ \max\limits_{j}\{d -\sum\limits_{i\neq j}m_i \}] 
\]
if $\pi_{5,0}(x) \in \bigcup_{m}(t_m \cup t_m')$. Now if $\pi_{5,0}(x) \in \bigcup\limits_{i,j,k,m}(\mathcal{Q}_{i,j,k} \setminus(t_m\ \cup \  t_m'))$ then take 
$\{i,j,k,m\} \subset \{1,2,...,5\}$ such that $\pi_{5,0}(x) \in \mathcal{Q}_{i,j,k} \setminus(t_m\ \cup \  t_m')$. Let $l$ be the line containing $\pi_{5,0}(x)$ and intersecting $l_i,l_j,l_k$ transversely. The Seshadri ratio of $\tilde{l}$ is then $d -m_i-m_j-m_k$, therefore $(3)$ follows. 
\end{proof}
\begin{xrem}
    Based on the proof of previous proposition (under the same assumption), it is evident that
    \[
    \min\left\{ m_L, d - \sum\limits_{i=1}^5 m_i \right\} \leq \varepsilon(X_{5,0}^3, L) \leq \min\left\{m_L, \min\limits_{j}\{d -\sum\limits_{i\neq j}m_i\}\right\}.
    \]
\end{xrem}
    
\section{Line blow-ups of $\mathbb{P}^4$ and $\mathbb{P}^5$}\label{Line blow-ups of}
    
The following propositions are Proposition 5.1 and Proposition 5.2 of \cite{PP}. We modify and re-write the propositions to make it convenient for our next results on Seshadri constants.
\begin{prop}[\cite{PP}]\label{proposition-3}
Let $X_{r,0}^4$ be a blow-up of $\mathbb{P}^4$ at $r$ general lines $l_1,..., l_r$. Then for $r \leq 7$, the nef cone of divisors on $X_{r,0}^4$ are as follows:
\begin{enumerate}  
\item $\Nef(X_{1,0}^4)= \{ dH - m_1E_1 \,  | \,   d \geq m_1, \, d \geq 0,  \,  m_1 \geq 0 \}$ 
\item $\Nef(X_{2,0}^4) = \{ dH - \sum_{i=1}^2 m_iE_i \,  | \,  d\geq m_1 +m_2, \, d \geq 0, \, m_i \geq 0   \, \text{for}  \, i =1,2 \}$ 
\item $\Nef(X_{r,0}^4) = \{ dH - \sum_{i=1}^r m_iE_i \,  | \,  d\geq m_i +m_j + m_k, \, i<  j<  k, \,  i, j, k \in \{1,..., r\}, \, d \geq 0, \, m_i \geq 0 \; \forall \; i \in \{1,...,r\} \} $ for $3 \leq r \leq 7$.
\end{enumerate}  
\end{prop}
  
\begin{proof}
Note that the inequalities $d \geq 0$ and $m_i \geq 0$ for $i=1,..., 7$ follow similarly as in the proof of Theorem \ref{theorem-2}.
Now, by the proof of Proposition 5.1 in \cite{PP}, nef cones of divisors on $X_{r,0}^4$ for $r \leq 7$ are generated by $H$, $H-E_i$ (for $i=1,...,r$) and $3H - \sum_{i=1}^r E_i$. Furthermore, we know that for any three lines in $\mathbb{P}^4$, there exists a line intersecting all three of them. So, the cases $r=1,2$, and $3$ are clear. For $r=4$, assume $m_1 = \min \{m_1, ..., m_4 \}$ after re-arranging the lines if needed. Write any divisor $D= dH-\sum_{i=1}^4 m_iE_i$ as  
 \[
 dH-\sum_{i=1}^4 m_iE_i =   m_1(3H - \sum_{i=1}^4 E_i) + (d-3m_1)H + \sum_{i=2}^4 (m_i - m_1) E_i. 
 \]
Note that, $3H - \sum_{i=1}^4 E_i$ is nef from Proposition 5.1 in \cite{PP}. The other part i.e., $(d-3m_1)H + \sum_{i=2}^4 (m_i - m_1) E_i$ is nef by $r=3$ case. So, the nef cone in this case can be written as $\Nef(X_{4,0}^4) = \{ dH - \sum_{i=1}^r m_iE_i \,  | \,  d\geq m_i +m_j + m_k, \, i <  j <  k, \,  i, j, k \in \{1,..., 4\}, \, d \geq 0, \, m_i \geq 0 \; \forall \, i \in \{1,...,4\} \}$. 
 Other cases i.e., $r=5,6$, and $7$ follows similarly by an iterative process.
  \end{proof}
\begin{prop}[\cite{PP}]\label{nef-cone-of-P5}
  Let $X_{r,0}^5$ be a blow-up of $\mathbb{P}^5$ at $r$ general lines $l_1,..., l_r$. Let $D = dH - \sum_i^r m_i E_i$ be any divisor on $X_{r,0}^5$. Then for $r \leq 5$, the nef cones of divisors on $X_{r,0}^5$ are described as:
\[  
\Nef(X_{r,0}^5) = \{dH - \sum_{i=1}^r m_iE_i \,  |  \, d \geq m_i + m_j, \,  i \neq j,\,  i, j \in \{1,...,r\}, \, d \geq  0, \, m_i \geq 0, \, \forall \, i\in \{1,..., r\}\}.
\]
\end{prop}

\begin{proof}
For $r\leq 5$, the nef cones $\Nef(X_{r,0}^5)$ are generated by $H$, $H - E_i$ (for $i = 1,..., r$) and $2H - \sum_{i=1}^r E_i$ (see Proposition 5.2 in \cite{PP}). From that we can describe the nef cones by following the method of the previous Proposition.
\end{proof}
Next, we examine the Seshadri constants of ample line bundles on $X_{r,0}^k$ ($k=4,5$). 
For simplification, we denote by $J_{j,k}^i := J(l_i,l_j) \cap J(l_i,l_k)$ to be the intersection of 
joins of $l_i,l_j$ and $l_i,l_k$. 

\begin{thm}\label{Seshadri-constant-on-lineblowup-of-P4}
Let $\pi_{r,0} : X_{r,0}^4 \rightarrow \mathbb{P}^4$ denote a blow-up of $\mathbb{P}^4$ along $r \leq 3$ general lines and $x$ be a general point of $X_{r,0}^4$. Let 
$L = dH - \sum\limits_{i=1}^{r}m_iE_i$ be an ample line bundle on $X_{r,0}^4$, then 
\begin{enumerate}
    \item $\varepsilon(X_{1,0}^4, L, 1) = d - m_1$ 
    \item $\varepsilon(X_{2,0}^4, L, x) \geq d - \sum\limits_{i=1}^2 m_i$, with equality if   $\pi_{2,0}(x) \in J(l_1,l_2)$ 
    \item $\varepsilon(X_{3,0}^4, L, x) \in [d-\sum\limits_{i=1}^3 m_i, \ \max\limits_{j} \{d - \sum\limits_{i\neq j} m_i\}]$, if $\pi_{3,0}(x) \in J_{2,3}^1  \cup J_{1,3}^2  \cup J_{1,2}^3$.
\end{enumerate}
\end{thm}

\begin{proof}
Let $x \in X_{r,0}^4$ be a general point and let $B$ be a reduced and irreducible curve in $X_{r,0}^4$ passing through $x$ with multiplicity $m$. Since $x$ is assumed to be a general point, we can assume that $x$ does not lie on any exceptional divisors. So, there is a curve $C$ in $\mathbb{P}^4$ such that $B = \tilde{C}$. Assume that $C$ has multiplicities $n_i$ along $l_i$ and 
let $\tilde{l}$ be the pullback of a general line in $\mathbb{P}^4$ and $\tilde{l_i}$ be the class of a line inside $E_i$ for $i=1,..., r$. In each of the following cases, after re-indexing the points $x_1^i,x_2^i, ..., x_{k_i}^i$ if necessary, we assume that $ n_j^i \geq n_{j+1}^i$ for $j = 1,..., k_i-1$. We also re-index the lines $l_1,...,l_r$ if necessary, to assume that $n_i \geq n_{i+1}$ for $i=1,2,...,r-1$.
\\
$\bullet$ $\underline{r=1, 2, 3}$:
 Arguing Similarly to the $r=1, 2, 3$ cases of Theorem \ref{Seshadri-constant-on-lineblowup-of-P3} we obtain the following inequalities:
\[ \varepsilon(X_{r,0}^4, L, 1)   \geq   d - \sum\limits_{i=1}^{r}m_i, \, \,    r = 1,2,3. \]
For $r=1$ consider the line joining $\pi_{1, 0}(x)$ and $x_{1}^1$ whose strict transform is $\tilde{l} - \tilde{l_1}$ and the corresponding Seshadri ratio is $d-m_1$ therefore in this case 
\[
\varepsilon(X_{1,0}^4, L, 1) = d-m_1.
\]
In case of $r=2$, if the point $\pi_{2,0}(x)$ is in a hyperplane of $\mathbb{P}^4$ i.e., $\pi_{2,0}(x) \in J(l_1,l_2) \cong \mathbb{P}^3$, we choose the line (denoted by) $l_x$ to be the line containing $\pi_{2,0}(x)$ and intersecting the given lines $l_1$, $l_2$. Clearly, 
$\tilde{l_x} = \tilde{l} - \tilde{l_1}-\tilde{l_2}$ and its Seshadri ratio is $d-m_1-m_2$. Therefore, in this case we get 
\[
\varepsilon(X_{2,0}^4, L, x) = d-m_1-m_2.
\]
For $r=3$, we note that if $\pi_{3,0}(x) \in J_{j,k}^i$ (for some $i,j,k$) 
then there exists a line, say $l_x$, containing $\pi_{3,0}(x)$ and intersecting $l_i,l_j$ and another line $l_x'$, containing $\pi_{3,0}(x)$ and intersecting $l_i,l_k$ transversely. Therefore the Seshadri ratio of $\tilde{l_x} = \tilde{l} - \tilde{l_i}-\tilde{l_j}$ and $\tilde{l_x'} = \tilde{l} - \tilde{l_i}-\tilde{l_k}$ are $d-m_i-m_j$ and $d-m_i-m_k$ respectively. Hence,  
\[
\varepsilon(X_{3,0}^4, L, x) \leq \max\limits_{j} \left\{d - \sum\limits_{i\neq j} m_i\right\},
\]
which establishes ($3$).
\end{proof} 

\begin{xrem}\label{Remark-on-P4-blowup-at-3-general-line}
Note that in the above theorem, if $\pi_{3,0}(x) \in l_j \cap J(l_i,l_k)$, then there is a line 
$l_x$ containing $\pi_{3,0}(x)$ and intersecting $l_1,l_2$, and $l_3$ transversely. The Seshadri ratio of this line is $d-m_1-m_2-m_3$. Therefore, we get 
\[
\varepsilon(X_{3,0}^4, L, x) \leq d-m_1-m_2-m_3.
\]
Also, since the point $x \in E_j$, we observe that for all curves $B \subset E_j$ passing through $x$ with multiplicity $m$, we have:
\[
\varepsilon(X_{3,0}^4, L, x) = \frac{L \cdot B}{m} = \frac{m_j \text{deg}(B)}{m} \geq m_j.
\]
Whereas the Seshadri ratio of a line inside $E_j$ is $m_j$. It then follows that:
\[
\varepsilon(X_{3,0}^4, L,x) = \min \{d-m_1-m_2-m_3, m_j\},
\]
if $\pi_{3,0}(x) \in l_j \cap J(l_i,l_k)$.
\end{xrem}

For $r=4$ lines, we have the following result, and similar results can be observed for $r=5,6,$ and $7$ cases as well.
\begin{prop}\label{prop about P4}
Let $X_{4,0}^4$ denote a blow-up of $\mathbb{P}^4$ along $4$ general lines $l_1, l_2, l_3, l_4$ and $x$ be a general point of $X_{r,0}^4$. Let $L = dH- \sum\limits_{i=1}^{4}m_iE_i$ be an ample line bundle on $X_{4,0}^4$ such that $d \geq \sum\limits_{i=1}^{4}m_i$. Then 
\begin{enumerate}
    \item $\varepsilon(X_{4,0}^4, L, 1) \geq d -\sum\limits_{i=1}^{4}m_i$
    \item $\varepsilon(X_{4,0}^4, L, x) \in [d -\sum\limits_{i=1}^{4}m_i,\ \max\limits_{i,j}\{d -m_i-m_j \}]$, \ \text{if\ $\pi_{4,0}(x) \in \bigcup\limits_{i\neq j}J(l_i,l_j)$}.
\end{enumerate}
\end{prop}
\begin{proof}
    The proof of $(1)$ follows the same approach as for $r=1,2$ and $3$ cases in Theorem \ref{Seshadri-constant-on-lineblowup-of-P3}. To see $(2)$, note that if $\pi_{4,0}(x) \in J(l_i,l_j)$, then there exists a line $l_x \subset J(l_i,l_j)$ containing $\pi_{4,0}(x)$ with a Seshadri ratio of $d-m_i-m_j$.
\end{proof}

\begin{thm}\label{Sesh-cons-on-blow-up-of-P5}
Let $\pi_{r,0} : X_{r,0}^5 \rightarrow \mathbb{P}^5$ denote a blow of $\mathbb{P}^5$ along $r =1,2$ general lines 
and $x$ be a general point of $X_{r,0}^5$. Let 
$L = dH - \sum\limits_{i=1}^{r}m_iE_i$ be an ample line bundle on $X_{r,0}^5$, then 
\begin{enumerate}
 \item $\varepsilon(X_{1,0}^5, L, 1) = d - m_1$ 
 \item $\varepsilon(X_{2,0}^5, L, x) \geq  d - \sum\limits_{i=1}^2 m_i, \ \text{with equality if}\ \pi_{2,0}(x) \in J(l_1,l_2)$.
\end{enumerate}
\end{thm}

\begin{proof}
Using Proposition \ref{nef-cone-of-P5} it can be proved similarly as the Theorem \ref{Seshadri-constant-on-lineblowup-of-P4}.
\end{proof}
\begin{xrem}
Under the assumption that $L = dH - \sum\limits_{i=1}^{r}m_iE_i$ is ample and $d\geq \sum\limits_{i=1}^{r}m_i$, similar results to Proposition \ref{prop about P4} can be obtained for $r=3,4$, and $5$ line blow-ups of $\mathbb{P}^5$.
\end{xrem}

We end with the following remark about the global Seshadri constants $\varepsilon(X_{r,0}^k,L)$ for $k=4$, and $5$.

\begin{xrem}
    Under the same assumption as in Theorem \ref{Seshadri-constant-on-lineblowup-of-P4}, 
    Proposition \ref{prop about P4}, and Theorem \ref{Sesh-cons-on-blow-up-of-P5}, the following results 
    about the global Seshadri constants follow directly from the proofs of these respective results and Remark 
    \ref{Remark-on-P4-blowup-at-3-general-line} 
    \begin{enumerate}
        \item $\varepsilon(X_{r,0}^k, L) = \min\left\{m_L, d-\sum\limits_{i=1}^{r} m_i \right\}$, \text{for} 
        $(k,r) \in \{4\} \times \{1,2,3\} \cup \{5\} \times \{1,2\} $\\
        \item $\min\left\{m_L, d-\sum\limits_{i=1}^r m_i \right\}\leq \varepsilon(X_{r,0}^k, L) \leq \min\left\{m_L,\min\limits_{i\neq j \in \{1,...,r\}} \{d- m_i-m_j\}\right\}$, \text{for} $(k,r) \in \{4\}\times \{4,5,6,7\} \cup \{5\}\times \{3,4,5\} $.
        \end{enumerate}
\end{xrem}

\subsection*{Acknowledgement}
The authors would like to thank Prof. Krishna Hanumanthu, CMI for his valuable suggestions and comments during the preparation of the manuscript.
%on a first draft of this article. 
The authors also thank the anonymous referee for reading the whole manuscript in details and giving multiple suggestions to improve the quality of the paper.
The first named author was supported partially by a grant from the Infosys Foundation initially and by NBHM (0204/1(18)/2022/ R\&D-II/1225) in the later stages of preparation of the manuscript. The second named author was funded 
by the postdoctoral fellowship of Shami Shamoon College of Engineering, Ashdod, Israel during the part of preparation of this manuscript.

\bibliographystyle{plain}

\end{document}